\documentclass{article}
\usepackage[english]{babel}
\usepackage[utf8x]{inputenc}
\usepackage[T1]{fontenc}
\usepackage{amsmath,amssymb,amsthm,mathtools,mathdots,nicefrac,tikz}

\usepackage[a4paper,top=3cm,bottom=2cm,left=3cm,right=3cm,marginparwidth=1.75cm]{geometry}

\usepackage{amsmath}
\usepackage{ dsfont }
\usepackage{amsthm}
\usepackage{amssymb}
\usepackage{graphicx}
\usepackage{caption}
\usepackage[colorlinks=true, allcolors=blue]{hyperref}
\usepackage{subcaption}

\theoremstyle{definition} 
\newtheorem{theorem}{Theorem}[section]
\newtheorem{construction}[theorem]{Construction}
\newtheorem{corollary}[theorem]{Corollary}
\newtheorem{lemma}[theorem]{Lemma}
\newtheorem{proposition}[theorem]{Proposition}

\newtheorem{conjecture}[theorem]{Conjecture}

\newcommand{\Z}{\mathbb{Z}}

\DeclareMathOperator{\Div}{Div}

\DeclareMathOperator{\dgon}{dgon}
\DeclareMathOperator{\sdgon}{sdgon}

\begin{document}

\title{Computing higher graph gonality is hard}
\author{Ralph Morrison and Lucas Tolley}
\date{}

\maketitle

\begin{abstract}

In the theory of divisors on multigraphs, the $r^{th}$ divisorial gonality of a graph is the minimum degree of a rank \(r\) divisor on that graph.  It was proved by Gijswijt et al. that the first divisorial gonality of a finite graph is NP-hard to compute.  We generalize their argument to prove that it is NP-hard to compute the \(r^{th}\) divisorial gonality of a finite graph for all \(r\).  We use this result to prove that it is NP-hard to compute  \(r^{th}\) stable divisorial gonality for a finite graph, and to compute  \(r^{th}\) divisorial gonality for a metric graph. We also prove these problems are APX-hard, and we study the NP-completeness of these problems.

\end{abstract}

\section{Introduction}
\label{section:introduction}

Divisor theory on graphs provides a combinatorial analog of divisor theory on algebraic curves.  This was introduced on finite graphs in \cite{bn} through the lens of chip-firing games, and was extended to metric graphs in \cite{gk,mz} through the language of tropical rational functions.  A direct link between divisor theory on graphs and divisor theory on curves was established in \cite{MB}, allowing for algebro-geometric results to be proved through purely combinatorial means, as done in \cite{cdpr}.

Of particular interest have been various analogs of the gonality of an algebraic curve.  This can be defined as the minimum degree of a positive rank divisor on the curve, and equals the minimum degree of a map from the curve to a line.  For finite or metric graphs, the \emph{divisorial gonality} is the minimum degree of a positive rank divisor.  In the case of finite graphs, one can also study the \emph{stable divisorial gonality}, the minimum divisorial gonality of any subdivision of the graph.  It was proved in \cite{GSW} that divisorial gonality and stable divisorial gonality are both NP-hard to compute for finite graphs. The same result was proved for metric graphs in \cite{small2020}; as mentioned in \cite[Remark 3.6]{DSW}, this result can also be deduced from arguments in \cite{GSW}.

In this paper we consider the computational complexity of higher gonalities of graphs.  For a finite or a metric graph, the \emph{\(r^{th}\) divisorial gonality} is the minimum degree of a divisor of rank \(r\); and for a finite graph, the \emph{\(r^{th}\) stable divisorial gonality} is the minimum \(r^{th}\) divisorial gonality of any subdivision of the graph.

Throughout, the problems we consider take a graph and an integer \(k\) as input, and asks whether the relevant type of \(r^{th}\) gonality is bounded by \(k\). Our first main result is the following, which we obtain by adapting the construction from \cite[\S 3]{GSW}.

\begin{theorem}\label{theorem:r_hard_finite}
     For any positive integer \(r\), the \textsc{\(r^{th}\) Divisorial Gonality} problem is NP-hard.
\end{theorem}

As in \cite{GSW}, our construction behaves nicely under subdivisions, allowing us to obtain the following result.

\begin{theorem}\label{theorem:r_hard_stable}
    For any positive integer \(r\), the \textsc{\(r^{th}\) Stable Divisorial Gonality}  problem is NP-hard.
\end{theorem}

From here, we utilize \cite[Theorem 1.3]{DSW} to develop our result for metric graphs.  This answers in the affirmative an open problem posed in \cite[Remark 3.6]{DSW}.

\begin{theorem}\label{theorem:r_hard_metric}
  For any positive integer \(r\), the \textsc{\(r^{th}\) Metric Divisorial Gonality} problem is NP-hard.
\end{theorem}

Our paper is organized as follows.  In Section \ref{section:background} we present necessary background on divisor theory on graphs.  In Section \ref{section:higher_finite} we prove that \(r^{th}\) gonality is NP-hard on finite graphs.  This is extended to \(r^{th}\) stable gonality in Section \ref{section:higher_stable}, and to metric graphs in Section \ref{section:higher_metric}.  We prove results regarding APX-hardness and NP-completeness in Section \ref{section:complexity}

\medskip

\noindent \textbf{Acknowledgements.} The authors thank Professor Pamela Harris for suggestions and comments on an early draft of these results.  The authors were supported by NSF Grant DMS-2011743.

\section{Divisor theory on graphs}
\label{section:background}

In this paper we deal with both finite graphs and metric graphs; if unspecified, a graph is taken to be finite.  Our finite graphs are connected multigraphs \(G=(V,E)\), where \(V=V(G)\) is a finite vertex set and \(E=E(G)\) is a finite edge multiset.  Note that we allow multiple edges connecting a pair of vertices, but never an edge from a vertex to itself; we denote the collection of edges connecting \(v,w\in V(G)\) by \(E(v,w)\). The \emph{valence} \(\textrm{val}(v)\) of a vertex \(v\in V(G)\) is the number of edges incident to \(v\).  We say that a graph \(H\) is a \emph{subdivision} of a graph \(G\) if \(H\) can be obtained by iteratively introducing \(2\)-valent vertices in the middle of the edges of \(G\).  We say a set \(S\subset V(G)\) is an \emph{independent set} if no two elements of \(S\) are connected by an edge in \(G\).  The \emph{independence number of \(G\)}, denoted \(\alpha(G)\), is the largest possible size of an independent set for that graph.  Given a subset \(S\subset V(G)\), we let \(G[S]\) denote the \emph{subgraph induced  by \(S\)}, whose vertex set is \(S\) and whose edge set is the subset of \(E\) with both endpoints in \(S\). 

The \emph{divisor group of \(G\)}, denoted \(\Div(G)\), is the free Abelian group on the vertex set \(V(G)\).  An element \(D\in \Div(G)\) is called a \emph{divisor}, and can be written
\[D=\sum_{v\in V(G)}D(v)\cdot v,\]
where \(D(v)\in\mathbb{Z}\). The \emph{degree of \(D\)} is the sum of its coefficients, i.e. \(\deg(D)=\sum_{v\in V(G)}D(v)\).  We intuitively think of a divisor as a placement of (possibly negative) integer numbers of poker chips on the vertices of the graph, so that the degree is the total number of chips. If \(D(v)\geq 0\) for all \(v\in V(G)\), we say that \(D\) is \emph{effective}.  If \(D(v)<0\) for some vertex \(v\), we say that $v$ is \emph{in debt}.

Order the vertices of \(G\) as \(v_1,\ldots,v_n\), and let \(L\) be the \emph{Laplacian} of \(G\), i.e. the \(n\times n\) matrix with diagonal entries \(L_{ii}=\textrm{val}(v_i)\) and off-diagonal entries \(L_{ij}=-|E(v_i,v_j)|\). Treating elements of \(\Div(G)\) as integer-valued vectors, we say two divisors \(D,D'\in \Div(G)\) are \emph{equivalent} if \(D'=D-L\sigma\) for some \(\sigma\in\mathbb{Z}^{|V(G)|}\).  We then write \(D\sim D'\).

This equivalence can be phrased more intuitively in the language of \emph{chip-firing games}.  Given a divisor \(D\), we transform it into a new divisor by ``firing'' a vertex \(v\), which moves chips from \(v\) to its neighbors (one along each edge).  Then, \(D\sim D'\) if and only if we can obtain \(D'\) from \(D\) via a sequence of chip-firing moves; in particular, if  \(D'=D-L\sigma\), then \(\sigma\) encodes the number of times each vertex should be fired.

It is useful to think about firing multiple vertices simultaneously.  Given a subset \(U\subset V(G)\), we can transform a divisor \(D\) into an equivalent divisor \(D'\) by firing all the vertices in \(U\) simultaneously.  In Laplacian notation, we have \(D'=D-L\mathds{1}_U\), where \(\mathds{1}_U\) is the \(0\)-\(1\) vector with \(1\)'s corresponding to the elements of \(U\).  A key fact is that if \(D\sim D'\) with \(D\) and \(D'\) both effective, then there exists a sequence of subset-firing moves transforming \(D\) into \(D'\) such that every intermediate divisor is also effective \cite[Corollary 3.11]{db}.  In fact, slightly more is true, which we summarize with the following lemma.

\begin{lemma}\label{lemma:no_debt_increase}
    Suppose \(D,D'\in \Div(G)\), with \(D\sim D'\) and \(D'\) effective.  Then there exists a collection of subset-firing moves transforming \(D\) to \(D'\) such that no subset-firing move introduces debt, and no subset-firing move increases the debt on any vertex.
\end{lemma}

\begin{proof}
Consider a firing script \(\sigma\in \mathbb{Z}^{|V(G)|}\) transforming \(D\) to \(D'\).  Since \(\mathds{1}_{V(G)}\in \ker L\), we can scale \(\sigma\) by the all \(1\)'s vector, and so we may assume that \(\sigma(i)\geq 0\) for all \(i\), and that \(\sigma(i)=0\) for at least one \(i\). Let \(k=\max_i\sigma(i)\), and for \(1\leq \ell\leq k\) consider the sets
\[U_\ell=\{v_i\in V(G)\,|\, \sigma(i)\geq k+1-\ell\}.\]
That is, \(U_\ell\) is the set of vertices that are fired at least \(k+1-\ell\) times by \(\sigma\).  We then have that the sequence of subset-firing moves given by \(U_1,U_2,\ldots,U_k\) is the same firing script as \(\sigma\) \cite[Lemma 2.3]{JDBG}

Since \(U_i\subset U_{i+1}\) for all \(i\), once a vertex is fired, it is fired as part of every subsequent subset-firing move, and will never gain more chips. Since \(D'\) is effective, it follows that the only vertices that can be fired are those that are not in debt.  Thus no vertex ever has a debt that increases from a subset-firing move.  Similarly, no vertex can be newly put into debt in a subset-firing move, since it would never gain more chips after that.  This completes the proof.
\end{proof}

Given a divisor \(D\in \Div(G)\), we let \(|D|\) denote the set of all effective divisors equivalent to \(D\).  If \(|D|=\emptyset\), we say \emph{\(D\) has rank \(-1\)}, written \(r(D)=-1\).  Otherwise, we define the \emph{rank} of \(D\), written \(r(D)\), to be the maximum value \(r\) such that for all effective divisors \(E\) of degree \(r\) we have that \(|D-E|\neq \emptyset\).  In the language of chip-firing, the rank of a divisor is the maximum amount of additional debt the divisor can eliminate, regardless of where that debt is placed.  For \(r\geq 1\), we then define the \emph{\(r^{th}\) divisorial gonality} \(\dgon_r(G)\) to be the minimum degree of a rank \(r\) divisor on \(G\).  From there, the \emph{\(r^{th}\) stable divisorial gonality} \(\sdgon_r(G)\) is the minimum \(r^{th}\) divisorial gonality of any subdivision of \(G\).

We now briefly describe metric graphs, which will be considered in Section \ref{section:higher_metric}. A \emph{metric graph} $\Gamma$ is a topological space arising from a pair $(G,l)$, where $G$ is a finite graph and $l$ is a length function, assigning to each edge of $G$ a positive real number. 
Given a finite graph $G$, we can construct a metric graph $\Gamma(G)$ by assigning a length of 1 to each edge in $E(G)$. 

We can study divisor theory on metric graphs, where a divisor is now a \(\mathbb{Z}\)-linear combination of all the points on the metric graph, with the requirement that all but finitely many coefficients are zero; the notions of degree and effectiveness can be defined as before.  Equivalence of divisors is phrased in the language of \emph{tropical rational functions}, and leads us to notions of the rank of a divisor; we refer the reader to \cite{gk} for more details.
 For any \(r\geq 1\), the \(r^{th}\) divisorial gonality \(\dgon_r(\Gamma)\) is the minimum degree of a rank \(r\) divisor on a metric graph \(\Gamma\).

 For a finite graph \(G\), let $\sigma_k(G)$ be the $k^{th}$ uniform subdivision of $G$, obtained by replacing every edge in $G$ with a path consisting of \(k\) edges. In \cite{DSW}, the authors prove that the gonalities of $G$ and $\Gamma(G)$ need not be equal in general, but that  \[\textrm{dgon}_r(\Gamma(G)) = \min_{k\in\Z_{>0}} \textrm{dgon}_r (\sigma_k(G)).\] With this we can readily prove that, given a graph $G$, the \(r^{th}
 \) metric gonality of \(\Gamma(G)\) falls  between \(r^{th}
 \) divisorial gonality and \(r^{th}
 \) stable divisorial gonality of \(G\).

\begin{proposition}\label{proposition:dsw_corollary} For a  finite graph \(G\) and for \(r\geq 1\), we have
\[\textrm{sdgon}_r(G) \leq \textrm{dgon}_r(\Gamma(G)) \leq \textrm{dgon}_r(G).\]
\end{proposition}

\begin{proof}
By \cite[Theorem 1.3]{DSW}, \(\textrm{dgon}_r(\Gamma(G)) = \min_{k\in\Z_{>0}} \textrm{dgon}_r (\sigma_k(G))\). Since each \(\sigma_k(G)\) is a subdivision of \(G\), we have \(\textrm{sdgon}_r(G)\leq \textrm{dgon}_r(\Gamma(G))\).  Since \(G=\sigma_1(G)\), we have \(\textrm{dgon}_r(\Gamma(G)) \leq \textrm{dgon}_r(G)\).
\end{proof}

\section{Higher gonality on finite graphs}
\label{section:higher_finite}

For any fixed \(r\geq 1\), we define the \textsc{\(r^{th}\) Divisorial Gonality} problem as follows.

\begin{itemize}
  \item[] \textsc{\(r^{th}\) Divisorial Gonality}
  \item[] \textbf{Input:} a graph \(G=(V,E)\) and an integer \(k\leq r|V|\).
  \item[] \textbf{Question:} is \(\dgon_r(G)\leq k\)?
\end{itemize}

       Given some finite graph $G$, we will create a new graph whose $r^{th}$ gonality is directly related to the independence number of $G$. Because the independence number of a graph is NP-hard to compute, it will follow that that $r^{th}$ gonality is NP-hard to compute.
    
    \begin{construction} \label{rth_construction}
    Given a graph $G=(V,E)$, construct $G_r'=(V',E')$ as follows. Begin with a single vertex $T$ in $G_r'$. For each vertex $v\in V$, create 3 vertices in $G_r': v, v'$ and $T_v$. Let $M = r(3|V| + 2|E| + 1) + 1$.  Between each $T$ and $T_v$, add $M$ parallel edges. Between each $v$ and $v'$, add $M$ parallel edges. Between $v'$ and $T_v$, add $r+2$ edges. For each edge $e(u,v) \in E$, create 2 vertices in $G_r': e_v$ and $e_u$. In $G_r'$, connect $v$ to $e_v$ with $M$ edges and connect $u$ to $e_u$ with $M$ edges, for each $v$ and $e_v$ pair. Between each pair of $e_v$ and $e_u$ with \(e(v,u)\in E(G)\), add $r$ edges.  
    \end{construction}
    
    \begin{figure}[hbt]
        \centering
        \includegraphics[scale = 0.65]{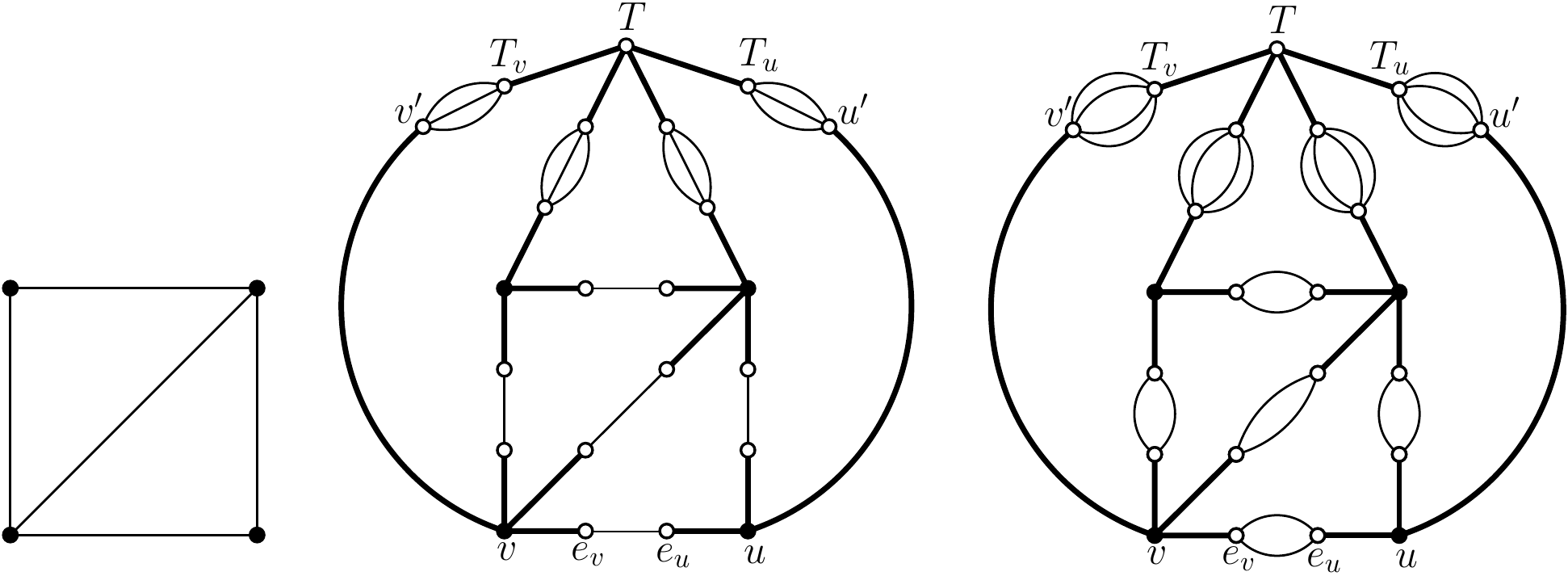}
        \caption{An example of $G$ and its corresponding $G_1'$ and $G_2'$. In each $G_r'$, bold edges represent $M$ parallel edges; note that \(M\) differs with varying \(r\).}
        \label{fig:rth_gon_construction}
    \end{figure}
    
    Note that $M$ was defined to be larger than $r$ times the number of vertices in $G_r'$. See Figure \ref{fig:rth_gon_construction} for an example how to construct $G_r'$. We remark that \(G_1'\) is precisely the graph \(\hat{G}\) constructed in \cite[\S 3]{GSW} to prove that first divisorial gonality is NP-hard to compute. Our analysis of \(G_r'\) will closely follow theirs.

    For a divisor $D$ on a graph, define the equivalence relation
    $$v \sim_D u \Leftrightarrow \sigma(v) = \sigma(u) \text{ for every } \sigma \in \mathbb{Z}^{|V|} \text{ for which } D-L\sigma \geq 0.$$
    An edge with two equivalent endpoints is called \emph{$D$-stopping} (this notion was originally called $D$-blocking in \cite{GSW}). Informally, this means that if in a chip-firing sequence one of the two endpoints is fired more times than the other one, then debt will be present somewhere in the graph.
    
    \medskip

    We now recall a few lemmas.
    
    \begin{lemma} \label{lemma:rth_gon_dstopping}
        (Lemma 3.1 in \cite{GSW}) Let $D \geq 0$ be a divisor on a graph $G$ and let $u, v \in V(G)$. Then $u \nsim_D v$ if and only if for some effective divisor $D' \sim D$ there exists a set \(U\subset V(G)\) with \(u\in U\), \(v\notin U\) such that \(D'-L\mathds{1}_U\geq 0\); that is, such that we can start with the divisor \(D'\) and subset-fire \(U\) without introducing debt. In particular, $u \sim_D v$ if every $u-v$ cut has more than $\deg(D)$ edges.
    \end{lemma}
    
\begin{corollary}\label{corollary:m_stop}
Let \(G'_r\) be constructed from \(G\) and \(r\) as in Construction \ref{rth_construction}, and let \(D\in \Div(G'_r)\) with \(r(D)=r\) and \(\deg(D)=\dgon_r(G'_r)\).  If \(x,y\in V(G'_r)\) with \(|E(x,y)|=M\), then \(x\sim_D y\).
\end{corollary}

\begin{proof}
Note that the divisor \(\sum_{v\in V(G)}rv\) has rank at least \(r\), which means that \(\deg(D)\leq r|V(G'_r)|\).  By construction, this is strictly smaller than \(M\).  It follows that every \(x-y\) cut has more than \(\deg(D)\) edges, so the result follows from Lemma \ref{lemma:rth_gon_dstopping}. 
\end{proof}

    \begin{lemma} \label{lemma:rth_gon_components}
        (Lemma 3.2 in \cite{GSW}) Let $D \geq 0$ be a divisor on $G = (V, E)$. Let $F$ be the set of $D$-stopping edges and let $U$ be a component of the subgraph $(V, E \setminus F)$. Then for every effective divisor $D' \sim D$ we have $\sum_{u\in U} D'(u) = \sum_{u\in U}D(u)$.
    \end{lemma}
    
    If we have two components of $(V', E' \setminus F)$ that are connected only by $D$-stopping edges, then the chips on each component must stay on that component. This sets the framework for the following results.
    
    For the next three lemmas, assume we are given a graph \(G\) and some \(r\geq 1\), and that we construct $G'_r$ as in Construction \ref{rth_construction}.  Assume further that $D\in \Div(G)$ is an effective divisor of rank $r$ and degree \(\deg(D)=\dgon_r(G'_r)\). By Corollary \ref{corollary:m_stop}, we know that \(T\sim_D T_v\) and \(v\sim_D v'\) for all \(v\in V(G)\).
    
    \begin{lemma} \label{lemma:rth_gon_Tv}
        For each $\{T_v,v'\}$ subgraph of $G'_r$, the number of chips of \(D\) on the subgraph must satisfy:
        \[
            D(\{T_v,v'\}) \geq \begin{cases}
                2r & \text{ if } T \sim_D v\\
                2r + 1 & \text{ if } T \nsim_D v.
            \end{cases}
        \]
    \end{lemma}
    
    \begin{proof} 
        If \(T\sim_D v\), then since we have $T \sim_D T_v$, and $v \sim_D v'$, we know from $T\sim_D v$ that we have $T_v\sim_D v'$. This means that no firing script can move chips from one of $T_v$ and $v'$ to the other without creating debt. Because $D$ has rank at least $r$, we claim that  $D(T_v) \geq r$ and $D(v')\geq r$. For if either \(D(T_v)\) or \(D(v')\) is smaller, then subtracting \(r\) chips from that vertex would create a divisor that could not be made effective, contradicting \(r(D)=r\). So if $T\sim_D v,$ then $D(\{T_V,v'\})\geq 2r$.
        
        If $T \nsim_D v$, then we have $T_v \nsim_D v'$. Then there exists a firing script $\sigma$ that fires $T_v$ at least one more time than $v'$, or $v'$ at least one more time than $T_v$, such that $D'=D - L\sigma \geq 0$. By Lemma \ref{lemma:no_debt_increase}, we can perform subset-firing moves without introducing any new debt, and so we must have at least $r+2$ chips placed on whichever vertex can fire more, one for each parallel edge connecting $T_v$ to $v'$. We will show that in order for \(D\) to have rank $r$, we must have $2r+1$ chips. Suppose that we only have $2r$ chips on the subgraph (the contradiction will work just as well if it is fewer). Without loss of generality, assume that $T_v$ is fired more than $v'$ by \(\sigma\). If $k \geq r+2$ chips are placed on $T_v$ and $2r-k \leq r-2$ chips are placed on $v'$, then \(D'(v')=3r-k\leq 2r-2\) and \(D'(T_v)=k-r\geq 2\). Place \(r\) debt, with $2r-k+1$ debt on $v'$ and the rest on $T_v$, resulting in $-1$ chips on $v'$ and at most $r+1$ chips on $T_v$. By Lemma \ref{lemma:no_debt_increase} and \(r(D')=r(D)\geq r\), we must be able to eliminate debt via subset-firing moves without introducing intermediate debt; but debt on \(v'\) can only be eliminated by firing \(T_v\) and not \(v'\), introducing debt on \(T_v\), a contradiction. Thus the subgraph must have at least $2r+1$ chips, so $D(\{T_v,v'\}) \geq 2r+1$.
    \end{proof}
    
    \begin{lemma} \label{lemma:rth_gon_ev}
         For each $\{e_v,e_u\}$ subgraph of $G_r'$ with \(e(v,u)\in E(G)\), the number of chips of \(D\) on the subgraph must satisfy:
        \[
            D(\{e_v,e_u\}) \geq \begin{cases}
                2r & \text{ if } v \sim_D u\\
                2r - 1 & \text{ if } v \nsim_D u.
            \end{cases}
        \]
    \end{lemma}
    
    \begin{proof}
        For the first case, note that $v \sim_D e_v$ and $u \sim_D e_u$. Thus if $v\sim_D u$ then we have $e_v\sim_D e_u$. This means that no firing script can move chips from one of $e_v$ and $e_u$ to the other without introducing debt. Because $D$ has rank $r$, we must have $D(e_v) \geq r$ and $D(e_u)\geq r$. Thus if $v\sim_D u, D(\{e_v,e_u'\})\geq 2r$.
        
        For the second case, $v\nsim_D u$, so $e_v\nsim_D e_u$. So, there exists some firing script \(\sigma\) with \(D'=D-L\sigma\) effective with \(e_v\) and \(e_v\) fired a different number of times.  In order not to introduce debt when one of the vertices is fired, there must be at least $r+1$ chips on one of the vertices of the subgraph, because there are $r+1$ edges between $e_v$ and $e_u$. Suppose that only $2r-2$ chips are placed on the subgraph. Similarly to Lemma \ref{lemma:rth_gon_Tv}, assume that $e_v$ is fired more than $e_u$. If $k \geq r+1$ chips are placed on $e_v$ and $2r-2-k \leq r-1$ chips are placed on $e_u$, we can place $2r-1-k$ chips of debt on $e_u$ and the rest on $e_v$. This results in $-1$ chips on $e_u$ and at most $r$ chips on $e_v$. Since $e_v$ must be part of any firing set that eliminates debt on \(e_u\), which would introduce debt on \(e_v\), $D$ does not have rank $r$. Thus the subgraph must have at least $2r-1$ chips, so $D(\{e_v,e_u\}) \geq 2r-1$.
    \end{proof}
    
    \begin{lemma} \label{lemma:rth_gon_alpha}
        Construct \(G'_r\) from \(G\) and \(r\) as in Construction \ref{rth_construction}.  We have $$\textrm{dgon}_r(G_r') = r + (3r+1)|V| + (2r-1)|E| - \alpha(G).$$
    \end{lemma}
     
    \begin{proof}
        First we prove that $\textrm{dgon}_r(G_r') \geq  r + (3r+1)|V| + (r+1)|E| - \alpha(G)$. Let $D$ be an effective divisor on $G$ of degree $\textrm{dgon}_r(G_r')$ and rank $r$. By Corollary \ref{corollary:m_stop}, each \(M\)-parallel edge is \(D\)-stopping. 
        Then for each $v \in V$, $T$ is equivalent to each $T_v$, $v$ is equivalent to $v'$, and $v$ is also equivalent to $e_v$. By Lemma \ref{lemma:rth_gon_components}, the number of chips on each component $\{T\}, \{T_v,v'\},\{v\},\{e_v,e_u\}$ must remain constant for each effective divisor $D'\sim D$. Then because $D$ has rank $r(D) \geq r$, for each $u,v \in V$ and $e\in E$ we must have
        \begin{itemize}
            \item $D(T) \geq r$;
            \item $D(v) \geq r$ for all $v\in V$;
            \item $D(T_v) + D(v') \geq 2r$ if $T\sim_D v$, otherwise $\geq 2r+1$ (by Lemma \ref{lemma:rth_gon_Tv}); and
            \item $D(e_v) + D(e_u) \geq 2r$ if $v\sim_D u$, otherwise $\geq 2r-1$ (by Lemma \ref{lemma:rth_gon_ev}).
        \end{itemize}

        Let $U_0 = \{v\in V ~|~ T \sim_D v\}$ be the set of all vertices in $V$ equivalent to $T$. Adding up the contributions above, we have that
        \begin{align*}
            \textrm{dgon}_r(G_r') &\geq r + r|V| +(2r+1)|V| - |U_0| + (2r-1)|E| + |E[U_0]|\\
            &= r + (3r+1)|V| + (2r-1)|E| - |U_0| + |E[U_0]|.
        \end{align*}
        Since  $\alpha(G) \geq \alpha(G[U_0]) \geq |U_0| - |E[U_0]|$, we have 
        $$\textrm{dgon}_r(G_r') \geq r + (3r+1)|V| + (2r-1)|E| - \alpha(G).$$
        
        Now to show equality, we must show there exists a divisor of rank at least $r$ and degree $$r + (3r+1)|V| + (2r-1)|E| - \alpha(G)$$
         on $G_r'$. Let $S$ an independent set on $G$ of size $\alpha(G)$. Let $v_1, \ldots, v_k$ be a numbering of the vertices of $V\setminus S$. Orient the edges of $G$ as follows: for $1 \leq i < j \leq k$, orient the edges in $E(v_i,v_j)$ from $v_i$ to $v_j$. For each edge between some $v_0\in S$ and $v_i$ for $1\leq i \leq k$, orient the edge from $v_0$ to $v_i$. Since \(S\) is an independent set, we have oriented all edges of \(G\). Define a divisor $D$ on $G_r'$ as follows:
        
        \begin{itemize}
            \item $D(T) = r$,
            \item $D(v) = r$ for $v \in V$,
            \item $D(T_v) = D(v') = r$ for $v\in S$,
            \item $D(T_v) = 2r+1$ for $v \notin S$,
            \item $D(v') = 0$ for $v \notin S$,
            \item $D(e_u) = 2r-1$ and \(D(e_v)=0\) for every edge $e(u,v) \in E$ with $u\in S$ and $v\notin S$ (note that \(u\) is the tail and \(v\) is the head), and
            \item $D(e_u) = r$ and $D(e_v) = r-1$ for every edge $e(v,u) \in E$ with $v,u\notin S$, with tail $u$ and head \(v\).
        \end{itemize}
        
        \begin{figure}[hbt]
            \centering
            \includegraphics[scale=0.65]{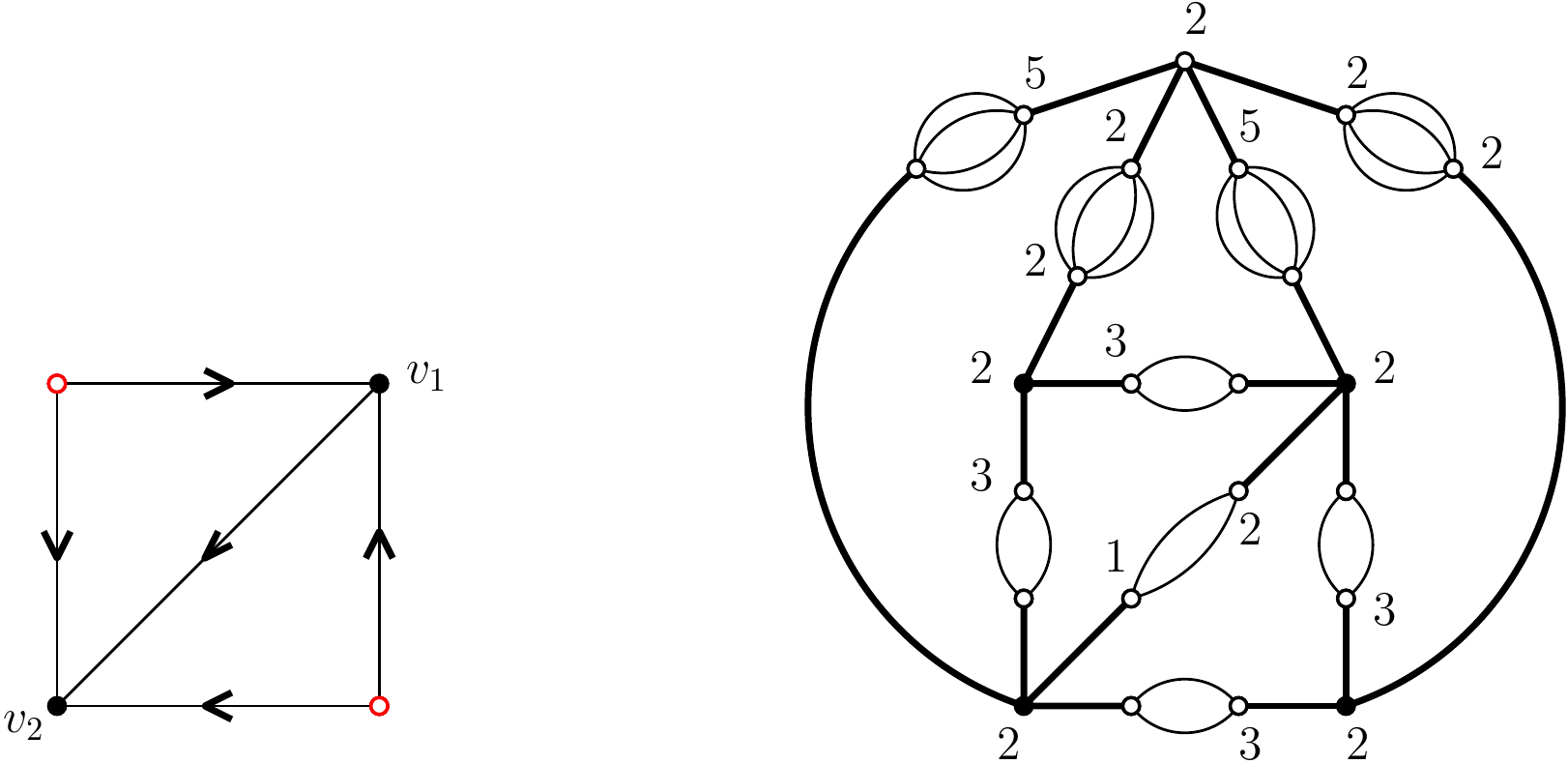}
            \caption{The left shows example ordering and orientation of $G$ with $S$ marked with empty circles. The right shows the corresponding divisor on $G_2'$, with degree $r + (3r+1)|V| + (2r-1)|E| - \alpha(G) = 43$.}
            \label{fig:rth_gon_divisor}
        \end{figure}
        
        Figure \ref{fig:rth_gon_divisor} shows an example of a graph \(G\) along with this divisor \(D\) on $G_2'$. The degree of $D$ is 
        $$r + (3r+1)|V| + (2r-1)|E| - \alpha(G).$$ Now we must show $D$ has rank at least $r$. Let $E$ be an effective divisor of degree $r$. If $D-E$ is effective, we are done. Assume \(D-E\) is not effective.

If \(D-E\) it has any debt on a \(\{e_v,e_u\}\) subgraph with \(v,u\notin S\), say with \(v\) as the head of \(e(u,v)\), then since \(D(e_v)=r-1\) and \(D(e_u)=r\) we know \(E=re_v\), and that \((D-E)(w)=D(w)\) for all \(w\neq e_v\); in particular, all debt in \(D-E\) is on \(e_v\).
To see that debt can be eliminated in this case, note that \(v=v_i\) for some \(i\), and let \(V_i=\{v_i,\ldots,v_k\}\). Consider the set
        $$W_i = V_i \cup \{v' ~|~ v\in V_i\} \cup \{e_v ~|~ v \text{ is an endpoint of the edge } e \text{ for some } v\in V_i\}.$$
        Transform \(D-E\) by firing \((W_i)^C\), the complement of this subset.  Since \(e_v\in W_i\) and \(e_u\notin W_i\), debt is eliminated on \(e_v\).  To verify that debt is not introduced elsewhere on the graph, note that the only net movement of chips is from \(T_v\) to \(v'\) where \(v\in V_i\subset V\setminus S\); and from \(e_u\) to \(e_v\) where \(e(u,v)\in E(G)\) with \(v\in V_i\) and \(u\notin V_i\), meaning by our choice of orientation that \(e(u,v)\) has tail \(u\) and head \(v\).  In the first case, \(T_v\) has \(2r+1\) chips and loses \(r+2\); and in the second case, since \(u\) is the tail, \(e_u\) has either \(2r-1\) or \(r\) chips and loses \(r\); thus no new debt is introduced.  Hence debt is eliminated in \(D-E\) by our subset-firing move. For an example of \(D-E\) and the subset-firing move in this case, see the left image in Figure \ref{figure:rth_gon_debt_update}.
        
        \begin{figure}[hbt]
            \centering
            \includegraphics[scale=0.6]{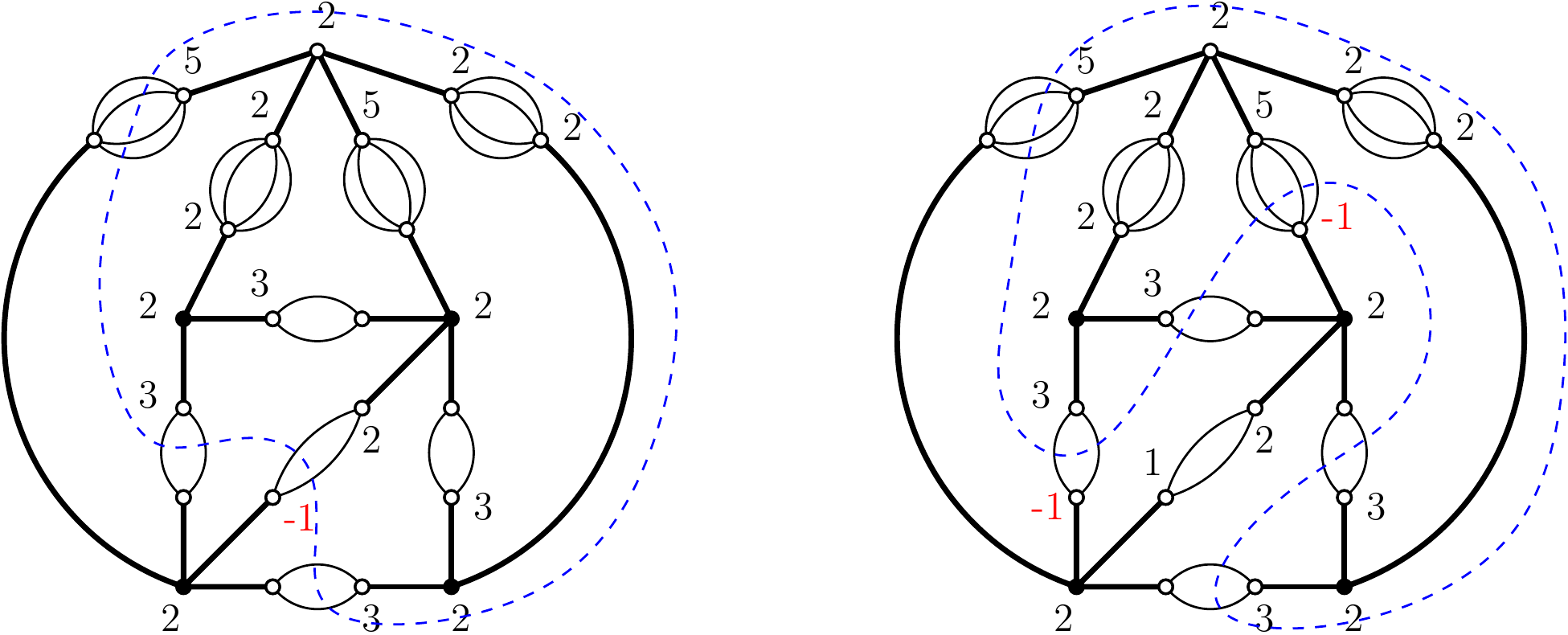}
            \caption{Two possible debt configurations of \(D-E\), and subsets to fire to eliminate the debt.}
            \label{figure:rth_gon_debt_update}
        \end{figure}
        
        Now assume that \(D-E\) has no debt on any \(\{e_v,e_u\}\) subgraph with \(v,u\notin S\).  No debt can be on a vertex \(w\) with \(D(w)\geq r\), so we now have that all debt in \(D-E\) is on a combination of vertices \(v'\) with \(v\notin S\), and vertices \(e_u\) with \(e(u,v)\in E(G)\) and \(u\notin S,v\in S\). Consider the following set of vertices:
        \[U=\{T\}\cup \{T_v\,|\, v\in V(G)\}\cup S\cup \{v'\,|\,v\in S\}\cup \{e_v\,|\, e(u,v)\in E(G), v\in S\}.\]
        First we claim that firing \(U\) does not introduce any new debt in \(D-E\).  Firing \(U\) has the following effect:
        \begin{itemize}
            \item \(r+2\) chips are moved from \(T_v\) to \(v'\), for every \(v\notin S\).
            \item \(r\) chips are moved from \(e_v\) to \(e_u\), for every \(e(u,v)\in E(G)\) with \(v\in S\).
        \end{itemize}
        Since \(\deg(E)=r\) and \(D-E\) is not effective, we have \((D-E)(w)\geq D(w)-(r-1)\) for every \(w\in V(G_r')\) with \((D-E)(w)\geq 0\).  Since \(D(T_v)=2r+1\) for \(v\notin S\), we have \((D-E)(T_v)\geq r+2\); and since \(D(e_v)=2r-1\) for \(v\in S\), we have \((D-E)(e_v)\geq r\). Thus firing \(U\) does not introduce any debt.  For an example of \(D-E\) and the subset-firing move in this case, see the right image in Figure \ref{figure:rth_gon_debt_update}.
        
        We now claim that firing the subset \(U\) eliminates all debt in \(D-E\). Indeed, at least \(r\) chips are moved to every \(v'\) with \(v\notin S\), and to every \(e_u\) with \(e(u,v)\in E(G)\) and \(v\in S\); as previously noted, these were the only possible vertices with debt, and each vertex had at most \(r\) debt since \(\deg(E)=r\).  Thus firing the subset eliminates all debt, and we have that \(r(D)\geq r\).  This lets us conclude that \[\dgon_r(G'_r)\leq \deg(D)=r + (3r+1)|V| + (2r-1)|E| - \alpha(G).\]
        
        Having obtained upper and lower bounds that are equal, we have shown that $\textrm{dgon}_r(G_r') = r + (3r+1)|V| + (2r-1)|E| - \alpha(G)$.
    \end{proof}
    
    We now conclude with the proof of Theorem \ref{theorem:r_hard_finite}, that computing $r^{th}$ divisorial gonality of a finite graph is NP-hard.
    
    \begin{proof}[Proof of Theorem \ref{theorem:r_hard_finite}.]
        Given $G$, we can construct $G_r'$ with Construction \ref{rth_construction} with only polynomial increase in the number of vertices and edges. By Lemma \ref{lemma:rth_gon_alpha}, $\textrm{dgon}_r(G_r') = r + (3r+1)|V| + (2r-1)|E| - \alpha(G)$. Because $\alpha(G)$ is NP-hard to lower bound, $r^{th}$ gonality is NP-hard to upper bound.
    \end{proof}

\section{Higher stable divisorial gonality}
\label{section:higher_stable}

     Recall that the \(r^{th}\) stable divisorial gonality of a graph $G$ is the minimum \(r^{th}\) divisorial gonality of any graph which is a subdivision of $G$.  For any fixed \(r\geq 1\), we define the \textsc{\(r^{th}\) Stable Divisorial Gonality} problem as follows.

\begin{itemize}
  \item[] \textsc{\(r^{th}\) Stable Divisorial Gonality}
  \item[] \textbf{Input:} a graph \(G=(V,E)\) and an integer \(k\leq r|V|\).
  \item[] \textbf{Question:} is \(\sdgon_r(G)\leq k\)?
\end{itemize}
     
     It is proved in \cite{GSW} that this problem is hard for \(r=1\). To prove $\textrm{sdgon}_r(G)$ is NP-hard for all $r$, we will use an additional definition and result from \cite{GSW}.
        Given an effective divisor $D$, define a \emph{$D$-stopping path} as a path where every internal vertex has degree 2 and whose ends are equivalent under $\sim_D$. Note that if $D'\sim D$ then the $D$-stopping paths are the same as the $D'$-stopping paths.
    
    A $D$-stopping path is considered \emph{clean} if it has 1 or fewer chips total on its internal vertices. Note that if two or more chips are on a path's internal vertices, subsets of the path can be fired to move at least one chip to an endpoint, so there is always an effective $D'\sim D$ in which all $D$-stopping paths are clean.
    
    \begin{lemma} \label{lemma:rth_sgon_dstopping}
        (Lemma 3.6 in \cite{GSW}) Let $D\geq0$ be a divisor on $G=(V,E)$. Suppose that all $D$-stopping paths are clean. Let $U$ be a component of the subgraph obtained from $G$ by deleting the edges and internal vertices of all $D$-stopping paths. Then for every effective divisor $D'\sim D$ we have $\sum_{u\in U} D'(u)\leq \sum_{u\in U} D(u)$.
    \end{lemma}
    
    We will now prove Theorem \ref{theorem:r_hard_stable}, that computing the $r^{th}$ stable gonality of a graph is NP-hard.
    
    \begin{proof}[Proof of Theorem \ref{theorem:r_hard_stable}.]
        Given a graph $G$, construct  $G_r'$ as in Construction \ref{rth_construction}. We prove that $\textrm{sdgon}_r(G_r') = \textrm{dgon}_r(G_r')$. For this, it is sufficient to prove $\textrm{dgon}_r(G_r'') \geq \textrm{dgon}_r(G_r')$ for an arbitrary subdivision $G_r''$ of $G_r'$.
        
        Let $G_r''$ be a subdivision of $G_r'$. Let $D$ be an effective divisor on $G_r''$ with rank $r$. Note that $r$ chips on each vertex of $G_r'$ has rank $r$ on $G_r''$, so $\deg(D) < M$, and thus the $D$-stopping edges in $G_r'$ are subdivided into $D$-stopping paths in $G_r''$. We assume that all $D$-stopping paths are clean. Then similar to Lemma \ref{lemma:rth_gon_alpha} (but using Lemma \ref{lemma:rth_sgon_dstopping} instead of Lemma \ref{lemma:rth_gon_dstopping}), for every vertex $v\in V$ and edge $e(u,v) \in E$, we have that
        \begin{itemize}
            \item $D(T) \geq r$,
            \item $D(v) \geq r$ for all $v\in V$,
            \item $D(T_v) + D(v') \geq 2r$ if $T\sim_D v$ ,
            \item $D([v']) \geq 2r+1$ if $T\nsim_D v$,
            \item $D(e_v) + D(e_u) \geq 2r$ if $v\sim_D u$,
            \item $D([e_v]) \geq 2r-1$ if $v\nsim_D u$,
        \end{itemize}
        where $[v']$ and $[e_v]$ are components of the subgraph of $G_r''$ after removing all interior edges and vertices of $D$-stopping paths. Like in Lemma \ref{lemma:rth_gon_alpha}, let $U_0 = \{v\in V ~|~ T \sim_D v\}$. Then we have that $\deg(D) \geq r + (3r+1)|V| + (2r-1)|E| - |U_0| + |E[U_0]|$.
        Since  $\alpha(G) \geq G[U_0] \geq |U_0| - |E[U_0]|$, we have $\deg(D) \geq r + (3r+1)|V| + (2r-1)|E| - \alpha(G) = \textrm{dgon}_r(G_r')$. Then $\textrm{dgon}_r(G_r'') \geq \textrm{dgon}_r(G_r')$, and so $\textrm{dgon}_r(G_r'') \geq \textrm{dgon}_r(G_r')$.  Since we chose an arbitrary subdivision, we have \(\sdgon_r(G_r')=\dgon_r(G_r')\). From here the remainder of the proof is identical to that of Theorem \ref{theorem:r_hard_finite}.
    \end{proof}

We can obtain our NP-hardness results for various subclasses of graphs.

\begin{corollary}
The \textsc{\(r^{th}\) Divisorial Gonality} and \textsc{\(r^{th}\) Stable Divisorial Gonality} remain NP-hard when we restrict the inputs to bipartite graphs.
\end{corollary}

\begin{proof}
Let \(G\) be any finite graph, let \(G'_r\) be constructed as usual, and consider \(\sigma_2(G'_r)\).  Note that \(\sigma_2(G'_r)\) is a bipartite graph, with partite sets given by the original vertices as one set, and the new vertices as the other set. Since this graph is a subdivision of \(G'_r\), by the previous proof we have \[\dgon_r(G'_r)=\sdgon_r(G'_r)\geq \dgon_r(\sigma_2(G'_r))\geq \dgon_r(G'_r),\]
so \(\dgon_r(G'_r)=\dgon_r(\sigma_2(G'_r))\).  Moreover, since any subdivision of \(\sigma_2(G'_r))\) is also a subdivision of \(G'_r\), we have \(\sdgon_r(G'_r)=\sdgon_r(\sigma_2(G'_r))\).  Thus if we can lower bound either \(\dgon_r(\sigma_2(G'_r))\) or \(\sdgon_r(\sigma_2(G'_r))\) efficiently, we can also upper bound \(\alpha(G)\) efficiently.  It follows that these problems are NP-hard, even for bipartite graphs.
\end{proof}

For the next result, we recall that a graph is an \emph{apex graph} if there exists a vertex that, when deleted, yields a planar graph.

\begin{theorem}
The \textsc{\(r^{th}\) Divisorial Gonality} and \textsc{\(r^{th}\) Stable Divisorial Gonality} remain NP-hard when we restrict the inputs to apex graphs.
\end{theorem}

\begin{proof}
Let \(G\) be a planar graph, and construct \(G'_r\) as usual.  We remark that \(G'_r\) is an apex graph, as deleting the vertex \(T\) yields a planar graph.  Since lower bounding \(\alpha(G)\) is NP-hard even for \(G\) planar \cite{independent_planar}, it follows that it must be NP-hard to upper bound \(r^{th}\) divisorial gonality and \(r^{th}\) stable divisorial gonality for apex graphs.
\end{proof}

We close this section with the following conjecture.

\begin{conjecture}
The \textsc{\(r^{th}\) Divisorial Gonality} and \textsc{\(r^{th}\) Stable Divisorial Gonality} remain NP-hard when we restrict the inputs to planar graphs.
\end{conjecture}

\section{Higher gonality on metric graphs}
\label{section:higher_metric}

For any fixed \(r\geq 1\), we define the \textsc{\(r^{th}\) Metric Divisorial Gonality} problem as follows.

\begin{itemize}
  \item[] \textsc{\(r^{th}\) Metric Divisorial Gonality}
  \item[] \textbf{Input:} a metric graph \(\Gamma=(G,l)\) and an integer \(k\leq r|V(G)|\).
  \item[] \textbf{Question:} is \(\dgon_r(\Gamma)\leq k\)?
\end{itemize}

Given a finite graph \(G\), recall that \(\Gamma(G)\) is the metric graph obtained by assigning edge lengths of \(1\) to each edge of \(G\).  Recall further that for \(k\geq 1\),  \(\sigma_k(G)\) denotes the finite graph obtained by subdividing each edge into \(k\) edges. 
We are now ready to prove that our metric problem is NP-hard.

\begin{proof}[Proof of Theorem \ref{theorem:r_hard_metric}] Let \(G=(V,E)\) be a finite connected graph, and let \(G'_r\) be the graph from Construction \ref{rth_construction}.  By Lemma \ref{lemma:rth_gon_alpha} and the argument of the proof of Theorem \ref{theorem:r_hard_stable}, we have that \(\sdgon_r(G'_r)\) and \(\dgon_r(G'_r)\) are both equal to
\[r + (3r+1)|V| + (2r-1)|E| - \alpha(G).\]
By Proposition \ref{proposition:dsw_corollary}, we know
\[\sdgon_r(G)\leq \dgon_r(\Gamma(G))\leq \dgon_r(G).\] Thus \(\dgon_r(\Gamma(G))\) is bounded above and below by the same number, and so must be equal to it.  It follows that if we can upper bound the \(r^{th}\) divisorial gonality of a metric graph efficiently, we can efficiently lower bound \(\alpha(G)\).  We conclude that the problem is NP-hard.
\end{proof}

\section{Other complexity classes}\label{section:complexity}

We close with results pertaining to other computational complexity classes, namely APX-hard problems and NP-complete problems.  We will prove that each variation of higher divisorial gonality is hard even to approximate; that is, that the \textsc{\(r^{th}\) Divisorial Gonality} problem, the \textsc{\(r^{th}\) Stable Divisorial Gonality}, and the \textsc{\(r^{th}\) Metric Divisorial Gonality} problem are all APX-hard.  Our proof closely follows those of \cite[Theorems 3.10 and 3.11]{GSW}, and relies on showing that being able to construct a ``good'' divisor of rank \(r\) (that is, one with degree close to the \(r^{th}\) gonality) allows one to quickly construct a ``good'' independent set (that is, one with a number of elements close to \(\alpha(G)\)).
We recall the following result.

\begin{lemma}[Corollary 3.8 in \cite{GSW}]\label{lemma:quick_equivalence} Given an effective divisor \(D\) on a graph \(G\) and any two vertices \(u,v\in V(G)\), we can determine in polynomial time if \(u\sim_Dv\).
    
\end{lemma}

Our next lemma is a generalization of \cite[Lemma 3.9]{GSW}.

\begin{lemma}\label{lemma:for_apx}
Let \(G\) be a subcubic graph, and let \(G_r'\) be as in Construction \ref{rth_construction}.  Let \(D\) be an effective divisor on \(G_r'\) of rank at least \(r\) with \(\deg(D)\leq (1+\varepsilon)\dgon_r(G_r')\).  Then we can construct in polynomial time an independent set on \(G\) of size at least \((1-(25r-3)\varepsilon)\alpha(G)\).
\end{lemma}

\begin{proof} We will assume that \(\deg(D)\leq r|V(G_r')|\); otherwise we may replace it with a divisor that has \(r\) chips on each vertex.

As in the proof of Lemma \ref{lemma:rth_gon_alpha}, let \(U_0\) be the set of vertices in \(V(G_r')\) that are equivalent to \(T\); by Lemma \ref{lemma:quick_equivalence}, we can find \(U_0\) in polynomial time. By the proof of Lemma \ref{lemma:rth_gon_alpha}, we know
\[\deg(D)\geq r+(3r+1)|V|+(2r-1)|E|-|U_0|+|E(U_0)|.\]
We construct an independent set \(S\) as follows:  start with the set \(U_0\), and for every edge in \(G[U_0]\) delete one of its endpoints.  This gives an independent set of size at least \(|U_0|-|E(U_0)|\).  Thus, we have
\begin{align*}
    |S|\geq&|U_0|-|E(U_0)|
    \\\geq &r+(3r+1)|V|+(2r-1)|E|-\deg(D)
    \\\geq &r+(3r+1)|V|+(2r-1)|E|-(1+\varepsilon)\dgon_r(G)
    \\=&r+(3r+1)|V|+(2r-1)|E|-(1+\varepsilon)(r+(3r+1)|V|+(2r-1)|E|-\alpha(G))
    \\=&\alpha(G)-\varepsilon(r+(3r+1)|V|+(2r-1)|E|-\alpha(G)).
\end{align*}
Since \(G\) is subcubic, we have \(|V|\leq4\alpha(G)\) and \(|E|\leq \frac{3}{2}|V|\leq 6\alpha(G)\).  We therefore have
\begin{align*}|S|\geq& \alpha(G)-\varepsilon(r+4(3r+1)\alpha(G)+6(2r-1)\alpha(G)-\alpha(G))
\\=& \alpha(G)-\varepsilon(r+(24r-3)\alpha(G))
\\\geq& \alpha(G)-\varepsilon(r\alpha(G)+(24r-3)\alpha(G))
\\=&(1-(25r-3)\varepsilon)\alpha(G).
\end{align*}
Thus we can construct, in polynomial time, an independent set of the claimed size.
\end{proof}

\begin{theorem}
For any fixed \(r\), the following problems are all APX-hard: \textsc{\(r^{th}\) Divisorial Gonality}, \textsc{\(r^{th}\) Stable Divisorial Gonality} , and  \textsc{\(r^{th}\) Metric Divisorial Gonality}.
\end{theorem}

\begin{proof}
The independent set problem is known to be APX-hard, even for subcubic graphs \cite{apx_cubic}.  Let \(G\) be a subcubic graph, and construct \(G_r'\) as usual. Using Lemma \ref{lemma:for_apx} and recalling that \(r\) is fixed, we find that \textsc{\(r^{th}\) Divisorial Gonality} is APX-hard.

Since we have
\[\dgon_r(G_r')=\sdgon_r(G_r')=\dgon_r(\Gamma(G_r')),\]
Lemma \ref{lemma:for_apx} holds if we replace \(\dgon_r(G_r')\) with either of \(\sdgon_r(G_r')\) or \(\dgon_r(\Gamma(G_r'))\).  The same  argument then lets us conclude that \textsc{\(r^{th}\) Stable Divisorial Gonality} and  \textsc{\(r^{th}\) Metric Divisorial Gonality} are APX-hard as well.
\end{proof}

We now turn to the question of NP-completeness.  Recall that a problem is NP-complete if in addition to being NP-hard, it is also in NP, meaning that a positive instance of the problem has a certificate that can be verified in polynomial time.

\begin{proposition}
The \textsc{$r^{th}$ Divisorial Gonality} problem is NP-complete.
\end{proposition}

\begin{proof}
By Theorem \ref{theorem:r_hard_finite}, this problem is NP-hard, so it remains to show that it is in NP.  Take as a certificate to a ``yes'' instance of the problem a divisor \(D\) of degree \(k\) with rank at least \(r\).  We must show that there exists a polynomial-time algorithm to check that \(r(D)\geq r\).

Let \(n=|V(G)|\).  We note that the number of effective divisors \(E\) of degree \(r\) on \(G\) is equal to the number of ways to place \(r\) identical objects into \(n\) distinct bins, namely \({n+r-1\choose n-1}\). Since \(r\) is fixed, we have \({n+r-1\choose n-1}=O(n^r)\); that is, there are polynomially many divisors of degree \(r\) on \(G\).  Thus it suffices to show that there is a polynomial time algorithm to check whether \(D-E\) is equivalent to an effective divisor for an arbitrary effective divisor \(E\) of degree \(r\). Such an algorithm is furnished by Dhar's burning algorithm \cite{dhar} or one of its modifications, see e.g. \cite[Corollary 6.5]{gonseq}.
\end{proof}

To our knowledge it is currently an open question whether  \textsc{\(r^{th}\) Stable Divisorial Gonality}  is in NP for all \(r\).  It was proved in \cite{stable_gonality_in_np} that the answer is ``yes'' for \(r=1\), and it may be that their techniques could be adapted to prove it for general \(r\).
In \cite[Remark 3.6]{DSW}, it is argued that for fixed \(r\) and \(k\), determining whether a metric graph \(\Gamma=(G,l)\) with rational edge lengths satisfies \(\dgon_r(\Gamma)\leq k\) is in NP. If one could push this argument to where \(k\) is not fixed, say to where it is bounded by \(r|V(G)|\), this could be used to argue that \textsc{\(r^{th}\) Metric Divisorial Gonality} is in NP for metric graphs of rational lengths.  These observations motivate the following conjectures.

\begin{conjecture}
The \textsc{\(r^{th}\) Stable Divisorial Gonality} problem is NP-complete.
\end{conjecture}

\begin{conjecture}
The \textsc{\(r^{th}\) Metric Divisorial Gonality} problem is NP-complete.
\end{conjecture}

\bibliographystyle{plain}

\begin{thebibliography}{10}

\bibitem{gonseq}
Ivan Aidun, Frances Dean, Ralph Morrison, Teresa Yu, and Julie Yuan.
\newblock Gonality sequences of graphs.
\newblock {\em SIAM J. Discrete Math.}, 35(2):814--839, 2021.

\bibitem{apx_cubic}
Paola Alimonti and Viggo Kann.
\newblock Some {APX}-completeness results for cubic graphs.
\newblock {\em Theoret. Comput. Sci.}, 237(1-2):123--134, 2000.

\bibitem{MB}
Matthew Baker.
\newblock Specialization of linear systems from curves to graphs.
\newblock {\em Algebra Number Theory}, 2(6):613--653, 2008.
\newblock With an appendix by Brian Conrad.

\bibitem{bn}
Matthew Baker and Serguei Norine.
\newblock Riemann-{R}och and {A}bel-{J}acobi theory on a finite graph.
\newblock {\em Adv. Math.}, 215(2):766--788, 2007.

\bibitem{stable_gonality_in_np}
Hans~L. Bodlaender, Marieke van~der Wegen, and Tom~C. van~der Zanden.
\newblock Stable divisorial gonality is in {NP}.
\newblock {\em Theory Comput. Syst.}, 65(2):428--440, 2021.

\bibitem{cdpr}
Filip Cools, Jan Draisma, Sam Payne, and Elina Robeva.
\newblock A tropical proof of the {B}rill-{N}oether theorem.
\newblock {\em Adv. Math.}, 230(2):759--776, 2012.

\bibitem{dhar}
Deepak Dhar.
\newblock Self-organized critical state of sandpile automaton models.
\newblock {\em Phys. Rev. Lett.}, 64(14):1613--1616, 1990.

\bibitem{small2020}
Marino Echavarria, Max Everett, Robin Huang, Liza Jacoby, Ralph Morrison, and
  Ben Weber.
\newblock On the scramble number of graphs.
\newblock {\em Discrete Applied Mathematics}, 310:43--59, 2022.

\bibitem{independent_planar}
M.~R. Garey and D.~S. Johnson.
\newblock The rectilinear {S}teiner tree problem is {NP}-complete.
\newblock {\em SIAM J. Appl. Math.}, 32(4):826--834, 1977.

\bibitem{gk}
Andreas Gathmann and Michael Kerber.
\newblock A riemann-roch theorem in tropical geometry.
\newblock {\em Mathematische Zeitschrift}, 259:217–230, 2007.

\bibitem{GSW}
Dion Gijswijt, Harry Smit, and Marieke van~der Wegen.
\newblock Computing graph gonality is hard.
\newblock {\em Discrete Appl. Math.}, 287:134--149, 2020.

\bibitem{mz}
Grigory Mikhalkin and Ilia Zharkov.
\newblock Tropical curves, their {J}acobians and theta functions.
\newblock In {\em Curves and abelian varieties}, volume 465 of {\em Contemp.
  Math.}, pages 203--230. Amer. Math. Soc., Providence, RI, 2008.

\bibitem{db}
Josse van Dobben~de Bruyn.
\newblock Reduced divisors and gonality in finite graphs.
\newblock Bachelor's thesis, Mathematisch Instituut, Universiteit Leiden, 2012.

\bibitem{JDBG}
Josse van Dobben~de Bruyn and Dion Gijswijt.
\newblock Treewidth is a lower bound on graph gonality.
\newblock {\em Algebr. Comb.}, 3(4):941--953, 2020.

\bibitem{DSW}
Josse van Dobben~de Bruyn, Harry Smit, and Marieke van~der Wegen.
\newblock Discrete and metric divisorial gonality can be different, 2021.

\end{thebibliography}

\end{document}